\documentclass[oneside,english]{amsart}

\usepackage{amssymb}
\usepackage{amscd}

\numberwithin{equation}{section} %% Comment out for sequentially-numbered
\numberwithin{figure}{section} %% Comment out for sequentially-numbered
\theoremstyle{plain}
\newtheorem*{thm*}{Theorem}
\theoremstyle{plain}
\newtheorem{thm}{Theorem}[section]
\theoremstyle{definition}
\newtheorem{defn}[thm]{Definition}
\theoremstyle{plain}
\newtheorem{lem}[thm]{Lemma}
\theoremstyle{plain}
\newtheorem{prop}[thm]{Proposition}
\theoremstyle{plain}
\newtheorem{cor}[thm]{Corollary}
\theoremstyle{remark}

\theoremstyle{remark}
\newtheorem*{acknowledgement*}{Acknowledgement}

\begin{document}

\title[A Finslerian version of Beltrami
Theorem]{A characterisation for Finsler metrics of constant curvature
  and a Finslerian version of Beltrami Theorem}

\author[Bucataru]{Ioan Bucataru}
\address{Faculty of  Mathematics \\ Alexandru Ioan Cuza University \\ Ia\c si, 
  Romania}
\email{bucataru@uaic.ro}
\urladdr{http://www.math.uaic.ro/\textasciitilde{}bucataru/}
\author[Cre\c tu]{Georgeta Cre\c tu}
\address{Faculty of  Mathematics \\ Alexandru Ioan Cuza University \\ Ia\c si, 
  Romania}
\email{Georgeta.Cretu@math.uaic.ro}

\date{\today}

\begin{abstract}
We define a Weyl-type curvature tensor that provides a characterisation for Finsler
metrics of constant flag curvature. When the Finsler metric reduces to a Riemannian metric, the
Weyl-type curvature tensor reduces to the classical projective Weyl tensor. In the general case, the Weyl-
type curvature tensor differs from the Weyl projective curvature, it is not a projective invariant,
and hence Beltrami Theorem does not work in Finsler geometry. We provide the relation between
the Weyl-type curvature tensors of two projectively related Finsler metrics. Using this formula
we show that a projective deformation preserves the property of having
constant flag curvature if and only if the projective factor is a Hamel
function. This way we provide a Finslerian version of Beltrami Theorem.
\end{abstract}



\maketitle

\section{Introduction}

For a Riemannian manifold of dimension greater than or equal to $3$, the projective Weyl curvature tensor vanishes
if and only if the Riemannian metric has constant sectional
curvature. Since the projective Weyl tensor is a projective invariant,
it follows that projective deformations preserve the property of
having constant sectional curvature and hence Beltrami Theorem is
true in Riemannian geometry. 

In Finsler geometry this theorem is no longer valid due to the nonlinearity of the projective factor. There
are examples of two projectively related Finsler metrics, which do not
have, both of them, constant curvature, \cite[3.9]{BCS00}, \cite{Shen02}.
In this paper we prove the following Finslerian version of Beltrami
Theorem (Theorem \ref{thm:FBT}):

\emph{Consider two projectively related Finsler
metrics and assume that one of them has constant
curvature. Then, the other Finsler metric has constant curvature if and
only if the projective factor is a Hamel function.}

The key ingredient in our work is a Weyl-type curvature tensor that provides in Theorem
\ref{thm:cfc}, in dimension greater than or equal to $3$, a characterisation for Finsler metrics of constant flag curvature.  We show that in Riemannian geometry the
projective factor is always a Hamel function and hence we provide a
Finslerian proof of classical Beltrami Theorem.  

For a Finsler manifold, of dimension greater than or equal to $3$, we define in \eqref{w0} a Weyl-type
curvature tensor that characterises Finsler metrics
of constant flag curvature. The Weyl-type curvature tensor is inspired by condition
$D_1$ of \cite[Theorem 4.1]{BM13} for characterising sprays
metrisable by Finsler metric of constant curvature.  This tensor
has been used also by Li and Shen \cite[(1.3)]{LS18} to characterise
sprays of isotropic curvature. The Weyl-type curvature tensor extends
the Riemannian projective Weyl curvature tensor and it reduces to the
Finslerian projective Weyl curvature tensor only in the case of constant flag
curvature. In the Finslerian context, the class of projective
deformations is larger than in the Riemannian case and we prove, in
Lemma \ref{lem:pw0}, that the Weyl-type curvature tensor is preserved
only by those projective deformations that satisfy the Hamel equation. There are some other proposals for Weyl-type curvature
tensors that characterise Finsler metrics of constant curvature,
\cite{AZ06, NT10, SM86}. For an extensive survey on the projective
geometry in the Riemannian case we refer to \cite{M18}.

\section{A geometric setting for sprays}

In this work $M$ represents a differentiable manifold of dimension greater than or equal to $3$. We denote by $TM$ the tangent bundle of $M$, while
$T_0M=TM\setminus 0$ denotes the tangent bundle with the zero section
removed. Local coordinates on the base manifold $M$ will be denoted by
$(x^i)$, while induced local coordinates on $TM$ or $T_0M$ will be
denoted by $(x^i, y^i)$. On $TM$ there is a canonical vector field,
the \emph{Liouville vector field}, and a canonical tensor field, the
\emph{tangent structure}, given by 
\begin{eqnarray*} {\mathcal C}=y^i\frac{\partial}{\partial y^i}, \quad
  J=\frac{\partial}{\partial y^i}\otimes  dx^i. \end{eqnarray*}

A vector field $S\in {\mathcal
    X}(T_0M)$ is called a \emph{spray} if it is a second order vector
  field, which means that $JS={\mathcal C}$, and positively
  $2$-homogeneous in the fiber coordinates, which
  means $[{\mathcal C}, S]=S$.
 
Locally, a spray $S \in {\mathcal X}(T_0M)$ can be represented as 
\begin{eqnarray}
S=y^i\frac{\partial}{\partial x^i} - 2G^i\frac{\partial}{\partial
  y^i}, \label{spray}
\end{eqnarray}
where the locally defined functions $G^i$ are positively
$2$-homogeneous with respect to the fiber coordinates. A regular curve
$c: t\in I \to (x^i(t))\in M$ is called a \emph{geodesic of the spray
  $S$}, given by formula \eqref{spray}, if $S\circ c'=c''$, which means that it satisfies the system of
second order ordinary differential equations:
\begin{eqnarray}
\frac{d^2x^i}{dt^2}+ 2G^i\left(x, \frac{dx}{dt}\right)=0. \label{sode}
\end{eqnarray}

A spray $S$, given by formula \eqref{spray}, is \textit{affine} if the
functions $G^i$ are quadratic in the fiber coordinates. In other
words, $G^i(x,y)=\gamma^i_{jk}(x)y^jy^k$, where $\gamma^i_{jk}$ are
the coefficients of an affine connection on the base manifold. For an
affine spray, its geodesics, solutions of the system \eqref{sode}, are
geodesics of the corresponding affine connection.

In this work, we use the Fr\"olicher-Nijenhuis theory of derivations
associated to various vector-valued forms defined on $TM$ or
$T_0M$, \cite[Chapter 2]{GM00}, to define a geometric setting for
studying sprays and Finsler spaces.

For a given spray $S$ we consider the induced geometric setting
consisting of: the horizontal projector $h$, the Jacobi endomorphism
$\Phi$ and the curvature tensor $R$ of the nonlinear connection, 
given by
\begin{eqnarray}
h=\frac{1}{2}\left(\operatorname{Id} - [S,J]\right), \quad
  \Phi = \left(\operatorname{Id} - h\right) \circ [S,h], \quad R=\frac{1}{2}[h,h]. 
\label{hphi}
\end{eqnarray} 
Locally, these geometric structures are given by 
\begin{eqnarray*}
h=\frac{\delta}{\delta x^i}\otimes dx^i, & \displaystyle \frac{\delta}{\delta
  x^i}= \frac{\partial}{\partial x^i} - \frac{\partial G^j}{\partial
  y^i} \frac{\partial}{\partial y^j}. \\
\Phi = R^i_j\frac{\partial}{\partial y^i}\otimes dx^j, & \displaystyle 
  R^i_j=2\frac{\partial G^i}{\partial x^j} -
  S\left(\frac{\partial G^i}{\partial y^j}\right) - \frac{\partial
  G^i}{\partial y^k}\frac{\partial G^k}{\partial y^j}. \\
R= R^i_{jk} \frac{\partial}{\partial y^i}\otimes dx^j \wedge dx^k, & \displaystyle 
  R^i_{jk} = \frac{\delta}{\delta x^k}\left( \frac{\partial G^i}{\partial
  y^j} \right) - \frac{\delta}{\delta x^j}\left( \frac{\partial G^i}{\partial
  y^k} \right).  
\end{eqnarray*}
The two curvature tensors, $\Phi$ and $R$, are related by the following formulae
\begin{eqnarray*} 
\Phi = i_SR, \quad 3R=[J, \Phi]. %\label{jphi} 
\end{eqnarray*}
For an affine spray $S$, the components of the Jacobi endomorphism are
quadratic in the fiber coordinates, $R^i_j(x,y)=R^i_{kjl}(x)y^ky^l$,
where $R^i_{kjl}$ are the curvature components of the affine
connection.

For a spray $S$ and a function $L$ on $T_0M$, consider the following
semi-basic $1$-form, called the \emph{Euler-Lagrange} $1$-form, 
\begin{eqnarray}
\label{elform} \delta_SL & =& {\mathcal L_S}d_JL - dL= d_J  {\mathcal
  L_S} L - 2d_hL  \\ & = &
\left\{S\left(\frac{\partial L}{\partial y^i}\right) - \frac{\partial
    L}{\partial x^i}\right\}dx^i = \left\{\frac{\partial
    S(L)}{\partial y^i} - 2\frac{\delta L}{\delta x^i}\right\}dx^i. \nonumber \end{eqnarray}
The function $L$ is called a Lagrangian and a spray $S$ that
satisfies the Euler-Lagrange equation $\delta_SL=0$ is called a geodesic spray.

\section{Finsler metrics and projective metrizability}

In Finsler geometry, the Lagrangian $L$ from the Euler-Lagrange
equation $\delta_SL=0$ is either a Finsler metric $F$ or the
kinetic energy $F^2/2$ of a Finsler metric.

\begin{defn} \label{def:Finsler}
A \emph{Finsler metric} is a continuous and positive function $F:TM\to {\mathbb R}$
that satisfies:
\begin{itemize}
\item[i)] $F$ is smooth on $T_0M$;
\item[ii)] $F$ is positively $1$-homogeneous in the fiber
  coordinates;
\item[iii)] the Hessian of the energy function 
\begin{eqnarray}
g_{ij}=\frac{1}{2}\frac{\partial^2 F^2}{\partial y^i \partial y^j} 
\label{gij}
\end{eqnarray}
is positive definite.
\end{itemize}
\end{defn}
The homogeneity of the Finsler metric and the Euler Theorem on
homogeneous functions implies that
$F^2(x,y)=g_{ij}(x,y)y^iy^j$ and $\frac{\partial F^2}{\partial
  y^i}=2g_{ij}y^j.$  When the metric tensor \eqref{gij} does not depend on the fiber coordinates, we say that the
Finsler metric reduces to a Riemannian metric.

For a Finsler metric $F$, the regularity condition iii) in the Definition
\ref{def:Finsler} assures that there is a unique spray $S$ satisfying the
Euler-Lagrange equation $\delta_SF^2=0$, which in view of formulae
\eqref{elform} is equivalent to $d_hF^2=0$.  We call $S$ the
\emph{geodesic spray} of the Finsler metric. The geodesics of the
Finsler metric are the geodesics of the corresponding spray $S$. 

A Finsler metric $F$ has \emph{scalar flag curvature} $\kappa \in
C^{\infty}(T_0M)$ if the Jacobi endomorphism, of its geodesic spray,
or the curvature of the nonlinear connection, are given by 
\begin{eqnarray}
\Phi=\kappa\left(F^2J-Fd_JF\otimes {\mathcal C}\right)
  \Longleftrightarrow R=\frac{1}{3F} d_J(\kappa F^3)\wedge J -
  \frac{F}{3} d_J\kappa \wedge d_JF \otimes {\mathcal C}. 
\label{sfc} 
\end{eqnarray}
Locally, Finsler metrics of scalar flag curvature $\kappa$ are characterised by
the following form of the Jacobi endomorphism
\begin{eqnarray*}
R^i_j = \kappa \left(F^2 \delta^i_j -
  \frac{1}{2}\frac{\partial F^2}{\partial y^j} y^i\right) =\kappa \left(g_{sl} \delta^i_j -
  g_{sj} \delta^i_l\right) y^s y^l. %\label{sfclocal} 
\end{eqnarray*}
When the scalar flag curvature $\kappa$ is a constant, we say that $F$
has constant flag curvature. Finsler metrics of constant curvature
$\kappa$ are characterised by the following form of the
curvature tensor
\begin{eqnarray}
R = \kappa F d_JF \wedge J. \label{cfc} 
\end{eqnarray}
Locally, for such Finsler metrics, the curvature tensor takes the form
\begin{eqnarray*}
R^i_{jl} = \kappa \left(g_{sl} \delta^i_j -
  g_{sj} \delta^i_l\right) y^s. %\label{cfclocal} 
\end{eqnarray*}
 
If $S$ is the geodesic spray of a Riemannian metric then the condition
\eqref{sfc} reduces to the isotropy condition $R^i_{sjl}=\kappa \left(g_{sl} \delta^i_j - g_{sj} \delta^i_l \right)$.

Two sprays (or Finsler metrics) are \emph{projectively related} if
their geodesics coincide as oriented non-parametrised curves. Two
sprays $\widetilde{S}$ and $S$ are projectively related if and only if
there exists a positively $1$-homogeneous function $P\in
C^{\infty}(T_0M)$ such that 
\begin{eqnarray*}
\widetilde{S}=S-2P{\mathcal C}. \end{eqnarray*}
A Finsler metric $\widetilde{F}$ is projectively related to a spray
$S$ if and only if it satisfies the \emph{Hamel equation}
$\delta_S\widetilde{F}=0$. In this case, the geodesic spray
$\widetilde{S}$ of $\widetilde{F}$ and $S$ are related by the
projective factor $P=S(\widetilde{F})/2\widetilde{F}$. 

There are various characterisations for the projective equivalence of
two sprays (Finsler metrics), the corresponding equations are known
as Rapcs\'ak equations, \cite[\S 9.2.3]{SLK14}. In the next theorem we
provide new characterisations for the projective metrizability that
include the projective factor $P$.

\begin{thm} \label{thm:pm} Consider a spray $S$ and a Finsler metric
  $\widetilde{F}$. Then the following relations are equivalent:
\begin{itemize}
\item[$PM_1)$] $\delta_S\widetilde{F}=0$;
\item[$PM_2)$] $d_h\widetilde{F}= d_J(P\widetilde{F})$;
\item[$PM_3)$] $\delta_S\widetilde{F}^2=2Pd_J\widetilde{F}^2$;
\item[$PM_4)$] $d_hd_J\widetilde{F}^2 = d_JP \wedge
  d_J\widetilde{F}^2$ and $S\widetilde{F}=2P\widetilde{F}$.
\end{itemize} 
\end{thm}

\begin{proof}
We assume that the relation $PM_1$ is true and hence the spray $S$ is
projectively related to the geodesic spray $\widetilde{S}$ of the
Finsler metric $\widetilde{F}$. In view of formulae \eqref{elform} we have that
$\delta_S\widetilde{F}=0$ is equivalent to
$d_JS\widetilde{F}=2d_h\widetilde{F}$. Since $S$ and $\widetilde{S}$
are projectively related, the projective factor is given by
$P=S\widetilde{F}/2\widetilde{F}$. We replace
$S\widetilde{F}=2P\widetilde{F}$  in the previous
equation and obtain $d_J(P\widetilde{F})=d_h\widetilde{F}$, which is relation
$PM_2$.

If we evaluate both sides of relation $PM_2$ on the spray $S$ we
obtain $S\widetilde{F}=2P\widetilde{F}$ and hence
$4P\widetilde{F}^2=S\widetilde{F}^2$.  Multiplying $PM_2$ by $4\widetilde{F}$ we obtain
$2d_h\widetilde{F}^2=4\widetilde{F}
d_J(P\widetilde{F}) = d_J(4P\widetilde{F}^2) - 2Pd_J\widetilde{F}^2$
and hence  $2d_h\widetilde{F}^2=  d_J(S\widetilde{F}^2) -
2Pd_J\widetilde{F}^2$. Again using formulae \eqref{elform}, last
formula can be written as
$\delta_S\widetilde{F}^2=2Pd_J\widetilde{F}^2$, which is relation
$PM_3$.

We start with relation $PM_3$ that is equivalent to $2d_h\widetilde{F}^2=  d_J(S\widetilde{F}^2) -
2Pd_J\widetilde{F}^2$. We apply $d_J$ to both sides of this relation,
use the commutation $[d_h, d_J]=0$ and $d_J^2=0$ and hence we obtain
the first relation $PM_4$. If we evaluate the semi-basic $1$-forms
from both sides of relation
$PM_3$ on the spray $S$ we obtain $S
\widetilde{F}^2=4P\widetilde{F}^2$, which gives the second relation $PM_4$.

Finally, we prove that the two relations $PM_4$ imply the projective
metrizability condition $PM_1$. From first relation $PM_4$ we obtain 
\begin{eqnarray*}
i_S d_hd_J\widetilde{F}^2 = i_S(d_JP \wedge
  d_J\widetilde{F}^2).
\end{eqnarray*}
We use the commutation formula $i_Sd_h=-d_hi_S + {\mathcal L}_{hS} +
i_{[h,S]}$ \cite[Appendix A]{GM00} and the fact that $J\circ [h,S]=-v$
to obtain
\begin{eqnarray*}
-2d_h\widetilde{F}^2 + {\mathcal L}_Sd_J\widetilde{F}^2 -
  d_v\widetilde{F}^2 = P d_J\widetilde{F}^2 - 2\widetilde{F}^2 d_JP. \label{eq1}
\end{eqnarray*}
We reformulate this formula in terms of the Finsler metric
$\widetilde{F}$ and use the second relation $PM_4$ to replace the projective 
factor $P$ and to obtain 
\begin{eqnarray*}
2S\widetilde{F} d_J\widetilde{F}  + 2 \widetilde{F}  {\mathcal
  L}_Sd_J\widetilde{F}  - 2 \widetilde{F} d\widetilde{F}  - 2
  \widetilde{F} d_h\widetilde{F} = S\widetilde{F} d_J\widetilde{F}  - 2\widetilde{F}^2d_J\left(\frac{S\widetilde{F}}{2\widetilde{F}}\right).
\end{eqnarray*}  
We divide both sides of the above formula by $2\widetilde{F}$ and
therefore
\begin{eqnarray*}
 {\mathcal  L}_Sd_J\widetilde{F}  - d\widetilde{F} = d_h \widetilde{F}
  - \frac{S\widetilde{F}}{2\widetilde{F}} d_J\widetilde{F} - \widetilde{F} d_J \left(\frac{S\widetilde{F}}{2\widetilde{F}}\right).
\end{eqnarray*}
Using formulae \eqref{elform}, the left hand side of the above
relation represents the semi-basic $1$-form $\delta_S\widetilde{F}$. The right hand side can
be written as $d_h \widetilde{F} -\frac{1}{2}
d_JS\widetilde{F}=-\frac{1}{2} \delta_S\widetilde{F}$. Therefore, last
formula implies
$\delta_S\widetilde{F}=-\frac{1}{2}\delta_S\widetilde{F}$ and hence we
proved that the Hamel condition $PM_1$ is satisfied.
\end{proof}

\begin{prop} \label{prop:drdj}
Consider two projectively related Finsler
metrics $F$ and $\widetilde{F}$, with the projective factor $P$. If $R$ is the curvature tensor of the Finsler metric $F$ then 
\begin{eqnarray}
d_Rd_J\widetilde{F}^2 = d_hd_JP \wedge d_J\widetilde{F}^2. \label{drdj1}
\end{eqnarray}
If the Finsler metric $F$ has constant curvature then 
\begin{eqnarray}
 d_hd_JP \wedge d_J\widetilde{F}^2=0. \label{drdj2}
\end{eqnarray}
If both Finsler metrics are reducible to Riemannian metrics
  and $F$ has constant curvature then the projective factor satisfies 
\begin{eqnarray}
 d_hd_JP =0. \label{drdj3}
\end{eqnarray}
\end{prop}
 
\begin{proof} Using the definition of the curvature tensor, last
  formula \eqref{hphi}, we have $d_R=\frac{1}{2}d_{[h,h]}=d^2_h$. If
  we apply $d_h$ of both sides of first formula $PM_4$ from Theorem
  \ref{thm:pm}, we obtain that formula \eqref{drdj1} is true. 

If the Finsler metric $F$ has constant curvature then the curvature
tensor $R$ is given by formula \eqref{cfc}. Using this, we have that 
\begin{eqnarray*}
d_Rd_J\widetilde{F}^2 =d_{\kappa F d_JF\wedge J} d_J\widetilde{F}^2 =
  \kappa F d_JF\wedge d_J d_J\widetilde{F}^2 = 0,
\end{eqnarray*}
and hence formula \eqref{drdj1} implies formula \eqref{drdj2}.

If both Finsler metrics $\widetilde{F}$ and $F$ are reducible to
Riemannian metrics, then their geodesic sprays are quadratic and hence
the projective factor $P$ is linear in the fiber coordinates. Assume
that $P(x,y)=a_i(x)y^i$. It follows that $d_JP=a_i(x)dx^i$ is a basic
$1$-form and hence 
\begin{eqnarray*} d_hd_JP=a_{ij}dx^i\wedge dx^j, \quad
  a_{ij}=\frac{1}{2}\left( \frac{\partial a_j}{\partial x^i} -
  \frac{\partial a_i}{\partial x^j} \right). 
\end{eqnarray*}  
If $F$ has constant curvature then formula \eqref{drdj2} is satisfied.
If the Riemannian metric $\widetilde{F}$ has the expression
$\widetilde{F}^2(x,y)=\widetilde{g}_{kl}(x)y^ky^l$ then formula
\eqref{drdj2} can be written as follows
\begin{eqnarray*}
\sum_{i, j, k \  cyclic}a_{ij}(x)\widetilde{g}_{kl}(x)y^l=0. 
\end{eqnarray*}
The linearity in $y$ of the above formula implies 
\begin{eqnarray*}
\sum_{i, j, k \ cyclic}a_{ij}(x)\widetilde{g}_{kl}(x)=0. 
\end{eqnarray*}
We multiply last formula by $\widetilde{g}^{ls}$ and obtain 
 \begin{eqnarray*}
a_{ij}\delta^s_k+ a_{ki}\delta^s_j + a_{jk}\delta^s_i = 0.
\end{eqnarray*}
In this formula we make the contraction $s=k$ and therefore
$a_{ij}=0$, which gives the desired conclusion, $d_hd_JP=0$. 
\end{proof}

\section{Weyl-type curvature tensor and Finsler metrics of constant curvature}

In Riemannian geometry, a metric has constant sectional curvature if
and only if the projective Weyl curvature tensor vanishes. Since the projective Weyl
tensor is a projective invariant it follows that projective deformations
between Riemannian metrics preserve the property of having constant
curvature, and this result is known as Beltrami Theorem.

In Finsler geometry, the projective Weyl tensor, which is a projective
invariant, characterises the metrics of scalar (not-necessarily constant)
flag curvature, \cite[Lemma 8.2.2]{Shen01}. There have been several
proposals for a Weyl-type curvature tensors, of $(1,3)$-type, that
characterise Finsler metrics of constant flag curvature, \cite{AZ06,
  NT10, SM86}, none of them being invariant under projective deformations. 

In this section, we define yet another Weyl-type curvature tensor, of $(1,1)$-type,
that characterises Finsler metrics of constant curvature in
Theorem \ref{thm:cfc}. The Weyl-type curvature tensor coincides with
the projective Weyl tensor if and only if the Finsler metric has
constant curvature and it reduces to the classical projective Weyl
tensor in the Riemannian context.  

Consider $S$ a spray with Jacobi endomorphism $\Phi$. Inspired by
\cite{BM13}, we define the
Weyl-type curvature tensor
\begin{eqnarray}
W_0=\Phi - \frac{1}{n-1} \left(\operatorname{Tr} \Phi\right) J +
  \frac{1}{2(n-1)} d_J\left(\operatorname{Tr} \Phi\right) \otimes
  {\mathcal C}. \label{w0} \end{eqnarray}
This tensor has been used in \cite[(1.20)]{LS18} to characterise
sprays with scalar curvature that does not depend on fiber
coordinates. Locally, the Weyl-type curvature tensor is given by 
\begin{eqnarray}
W_0=W^i_{0j}\frac{\partial}{\partial y^i}\otimes dx^j, \quad W^i_{0j}=R^i_j -
  \frac{1}{n-1}R^l_l\delta^i_j  + \frac{1}{2(n-1)}\frac{\partial
  R^l_l}{\partial y^j} y^i. \label{w0l}
\end{eqnarray}
The Weyl-type curvature tensor is traceless,
$\operatorname{Tr}(W_0)=W^i_{0i}=0$, and satisfies $W_0(S)=W^i_{0j}y^j
\frac{\partial}{\partial y^i} = 0$, for any spray $S$.

If the spray $S$ is affine, the Weyl-type curvature tensor \eqref{w0}
is quadratic in the fiber coordinates and can be
expressed as follows: 
\begin{eqnarray}
W_0=W^i_{jkl}(x)y^jy^l\frac{\partial}{\partial y^i}\otimes dx^k, \label{w0r}\end{eqnarray}
where $W^i_{jkl}$ is the classical projective Weyl tensor
\begin{eqnarray}
W^i_{jkl}=R^i_{jkl} - \frac{1}{n-1}\left(Ric_{jl}\delta^i_k - Ric_{jk}
  \delta^i_l \right), \label{rw} \end{eqnarray}
and $Ric_{ij}=R^l_{ilj}$ is the Ricci tensor. Formula \eqref{w0r} shows that in the Riemannian context one can
recover the projective Weyl tensor from the Weyl-type tensor
\eqref{w0}.  

The classical Weyl tensor in Finsler geometry, which is a projective
invariant, characterises Finsler metrics of scalar flag curvature,
\cite[Theorem 26.1]{M86},  and it is given by:
\begin{eqnarray}
W^i_j = R^i_j - \frac{1}{n-1} R^l_l \delta^i_j -
  \frac{1}{n+1} \frac{\partial}{\partial y^m}\left( R^m_j -
  \frac{1}{n-1} R^l_l \delta^m_j \right) y^i. \label{weyl}
\end{eqnarray}
Next theorem provides a characterisation for Finsler metrics of
constant curvature using the Weyl-type curvature tensor \eqref{w0}.
\begin{thm} \label{thm:cfc}
For a Finsler metric the following conditions are equivalent:
\begin{itemize}
\item[i)] The Finsler metric has constant flag curvature.
\item[ii)] The Weyl-type curvature tensor \eqref{w0} vanishes. 
\item[iii)] The Finsler metric has scalar flag curvature and the Weyl
  tensors \eqref{w0l} and \eqref{weyl} coincide.
\end{itemize}
\end{thm}
\begin{proof}
For the implication $i) \Longrightarrow ii)$, we assume that the Finsler
metric $F$ has constant flag curvature $\kappa$. Therefore, the Jacobi
endomorphism is given by formula \eqref{sfc},  with $\kappa$ constant. It follows that $\operatorname{Tr}
\Phi=(n-1)\kappa F^2$, $d_J\left(\operatorname{Tr}
  \Phi\right)=2(n-1)\kappa F d_JF$ and hence the Weyl-type curvature tensor
\eqref{w0} vanishes. 

For the implication $ii) \Longrightarrow i)$, we assume that
the Weyl-type curvature tensor \eqref{w0} vanishes, which means that the Jacobi
endomorphism is given by  
\begin{eqnarray}
\Phi = \frac{1}{n-1} \left(\operatorname{Tr} \Phi\right) J -
  \frac{1}{2(n-1)} d_J\left(\operatorname{Tr} \Phi\right) \otimes
  {\mathcal C}. \label{phi0} \end{eqnarray}
Consider $S$ the geodesic spray of the Finsler metric $F$, which
means that it satisfies $d_hF^2=0$. Using second formula \eqref{hphi} we obtain $d_{\Phi}F^2=0$. In the last formula, we
use the expression \eqref{phi0} of the Jacobi endomorphism to get
\begin{eqnarray}
\left(\operatorname{Tr} \Phi\right) d_JF^2 - F^2
  d_J\left(\operatorname{Tr} \Phi\right)=0. \label{djk}
\end{eqnarray}
If $\Phi=0$, which is equivalent to $\operatorname{Tr} \Phi=0$, we
have that $\kappa=0$. Assume now that $\operatorname{Tr} \Phi\neq 0$.

One can reformulate formula \eqref{djk}  as follows:
\begin{eqnarray*}
\frac{d_J F^2}{F^2} = \frac{ d_J\left(\operatorname{Tr}
  \Phi\right)}{\operatorname{Tr} \Phi}. 
\end{eqnarray*}
Therefore, the Jacobi endomorphism \eqref{phi0} can be written as
\begin{eqnarray*}
\Phi = \frac{1}{n-1} \operatorname{Tr} \Phi \left( J - \frac{ d_J\left(\operatorname{Tr}
  \Phi\right)}{2 \operatorname{Tr} \Phi} \otimes {\mathcal C}\right) =
  \frac{\operatorname{Tr} \Phi }{(n-1)F^2} \left(F^2  J - F
  d_JF\otimes {\mathcal C}\right). \end{eqnarray*}
If we consider the function 
\begin{eqnarray}
\kappa =\frac{\operatorname{Tr} \Phi }{(n-1)F^2}, \label{kf2}
\end{eqnarray}
the Jacobi endomorphism $\Phi$ is given by formula \eqref{sfc}. Using
formula \eqref{djk}, it follows that the function $\kappa$, defined by
formula \eqref{kf2}, satisfies $d_J\kappa=0$.  Therefore $\kappa$ is
constant along the fibers of the tangent bundle and, in view of the 
Finslerian version of the Schur Lemma \cite[Proposition 26.1]{M86}, we obtain that $\kappa$ is constant and hence the Finsler
metric has constant flag curvature. 

We prove now the equivalence of the two conditions $i)$ and $iii)$. The Weyl
tensors \eqref{w0l} and \eqref{weyl} coincide if and only if 
\begin{eqnarray*}
\frac{1}{2(n-1)} \frac{\partial R^l_l}{\partial y^j}y^i =
  \frac{-1}{n+1} \left( \frac{\partial R^m_j}{\partial y^m} -
  \frac{1}{n-1} \frac{\partial R^l_l}{\partial y^j} \right) y^i,
\end{eqnarray*}
which is equivalent to 
\begin{eqnarray}
  \frac{1}{2} \frac{\partial R^l_l}{\partial
  y^j}y^i = - \frac{\partial R^m_j}{\partial y^m}y^i. \label{t1}
\end{eqnarray}
We assume now  that the Finsler metric has scalar flag curvature,
which means that the Jacobi endomorphism is given by formula
\eqref{sfc}. Therefore, we have $R^m_j=\kappa F^2 \delta^m_j -
\frac{\kappa}{2} \frac{\partial F^2}{\partial y^j}y^m$  and
$R^l_l=(n-1)\kappa F^2$. For the left hand side of formula
\eqref{t1}, we have 
\begin{eqnarray}
\frac{\partial R^l_l}{\partial y^j} = (n-1) \frac{\partial \kappa}{\partial
  y^j} F^2 + (n-1)\kappa \frac{\partial F^2}{\partial y^j}. \label{t2}
\end{eqnarray}
Using the homogeneity of the terms, we have for the right side of
formula \eqref{t1}:
\begin{eqnarray}
\frac{\partial R^m_j}{\partial y^m} =  \frac{\partial \kappa}{\partial
  y^j} F^2 - (n-1)\frac{\kappa}{2} \frac{\partial F^2}{\partial y^j}. \label{t3}
\end{eqnarray}
Using formulae \eqref{t2} and \eqref{t3} we have that \eqref{t1} is
true if and only if $\frac{\partial \kappa}{\partial y^j} =0$, which, 
in view of the Finslerian version of the Schur Lemma \cite[Proposition
26.1]{M86}, it means that the Finsler metric has constant flag curvature.
\end{proof}
If the Finslerian metric reduces to a Riemannian metric then Theorem
\ref{thm:cfc} reduces to the following characterisation of Riemannian
metrics of constant curvature. 
\begin{cor} \label{csc}
A Riemannian metric has constant sectional curvature if and only if the Weyl-type
curvature tensor \eqref{w0} vanishes.  
\end{cor}
\begin{proof}
For an affine spray, the Weyl-type curvature tensor \eqref{w0} is given by
formula \eqref{w0r}. Therefore, if for a Riemannian metric $g$, the Weyl-type
curvature tensor \eqref{w0} vanishes, then the tensor \eqref{w0r}
vanishes as well. Since
$W_0$, in formula \eqref{w0r}, is quadratic in the fiber coordinates, then $W_0=0$ implies
$W^i_{jkl}+W^i_{lkj}=0$, the skew-symmetry in first and third
covariant indices. Using this symmetry and the properties of the projective Weyl
tensor (skew-symmetry in the last two covariant indices and
algebraic Bianchi identity) we obtain 
\begin{eqnarray*}
0=W^i_{jkl}+W^i_{ljk}+W^i_{klj}=W^i_{jkl}-W^i_{lkj}-W^i_{jlk}=3W^i_{jkl}.
\end{eqnarray*} 
Consequently, the projective Weyl tensor \eqref{rw} vanishes, which
means that the Riemannian curvature tensor is given by 
\begin{eqnarray*}
R^i_{jkl} =\frac{1}{n-1}\left(\delta^i_k Ric_{jl} - \delta^i_l
  Ric_{jk}\right), \end{eqnarray*}
which is equivalent to 
\begin{eqnarray*}
R_{ijkl} =\frac{1}{n-1}\left(g_{ik} Ric_{jl} - g_{il} Ric_{jk}\right). \end{eqnarray*}
The symmetry of the curvature tensor $R_{ijkl}=R_{klij}$ implies the
Ricci tensor $Ric_{ij}$ and the metric tensor $g_{ij}$ are
proportional, say $Ric_{ij}=(n-1)\kappa g_{ij}$. Therefore, the
curvature tensor is given by    
\begin{eqnarray*}
R^i_{jkl} =\kappa\left(\delta^i_k g_{jl} - \delta^i_l
  g_{jk}\right), \end{eqnarray*}  
which in view of Schur Lemma we have that $\kappa$ is constant and hence
the Riemannian metric has constant curvature $\kappa$.\end{proof}

Formula \eqref{kf2} shows how to recover the scalar curvature of a
Finsler metric from the fact that the Weyl-type
tensor \eqref{w0} vanishes. In the Riemannian context, formula \eqref{kf2} reduces to 
\begin{eqnarray*}
\kappa =\frac{Ric_{ij}y^iy^j }{(n-1)g_{ij}y^iy^j}= \frac{Ric_{ij}}{(n-1)g_{ij}}, 
\end{eqnarray*} which is true due to the fact that the vanishing of
the projective Weyl tensor \eqref{rw} implies $Ric_{ij}= (n-1)\kappa
g_{ij}$.

\section{A Finslerian version and a Finslerian proof of classical Beltrami Theorem}

Although the Weyl-type curvature tensor \eqref{w0} is not a projective
invariant, we find the class of projective deformations that preserve
Weyl-type curvature tensors and therefore we obtain the class of projective
deformations in Lemma \ref{lem:pw0}, which makes Beltrami Theorem work
in Finslerian setting, see Theorem \ref{thm:FBT}.  

\begin{lem} \label{lem:pw0}
The corresponding Weyl-type curvature tensors of two projectively related
sprays $S$ and $\widetilde{S}=S-2P{\mathcal C}$ are related by
\begin{eqnarray}
\widetilde{W_0}=W_0-\frac{3}{2}\delta_SP \otimes {\mathcal C}. \label{pw0}
\end{eqnarray}
\end{lem} 
\begin{proof}
For two projectively related sprays $S$ and $\widetilde{S}=S-2P{\mathcal C}$,
the corresponding Jacobi endomorphisms are related by the following
formula \cite[(4.8)]{BM12}: 
\begin{eqnarray*}
\widetilde{\Phi} = \Phi + (P^2-SP)J -\left(Pd_JP + d_JSP -
  3d_hP\right)\otimes {\mathcal C}. 
\end{eqnarray*}
This implies $\operatorname{Tr} \widetilde\Phi = \operatorname{Tr}
\Phi + (n-1) (P^2-SP)$. Using these formulae, the Weyl-type curvature tensor
\eqref{w0} of the spray
$\widetilde{S}$ can be written as follows: 
\begin{eqnarray*}
\widetilde{W_0}  & = & \Phi + (P^2-SP)J -\left(Pd_JP + d_JSP -
  3d_hP\right)\otimes {\mathcal C}  \\
 & - & \frac{1}{n-1} \left(\operatorname{Tr} \Phi\right) J -
                       (P^2-SP)J +  \frac{1}{2(n-1)}
                       d_J\left(\operatorname{Tr} \Phi\right) \otimes
                       {\mathcal C} + \frac{1}{2} d_J(P^2 - SP)
                       \otimes {\mathcal C}  \\ & = & W_0 - \frac{3}{2}
                                                      \left( d_JSP -
                                                      2d_hP\right)
                                                      \otimes
                                                      {\mathcal
                                                      C},
\end{eqnarray*} 
and hence formula \eqref{pw0} is true.
\end{proof}

We can formulate the following version of Beltrami Theorem in Finsler geometry.
\begin{thm} \label{thm:FBT} \emph{(Finslerian version of Beltrami Theorem)}
Consider two projectively related Finsler
metrics $F$ and $\widetilde{F}$. If one of the Finsler metrics has
constant curvature then the other one has also constant curvature if
and only if the projective factor $P$ satisfies either one of the
following equivalent Hamel conditions:
\begin{itemize}
\item[i)] $\delta_SP=0$,
\item[ii)] $d_hd_JP=0$.
\end{itemize}
\end{thm}
\begin{proof}
For two projectively related Finsler metrics $F$ and
$\widetilde{F}$, we consider the corresponding Weyl-type curvature
tensors $W_0$ and $\widetilde{W}_0$, which are related by formula
\eqref{pw0}. Using Theorem \ref{thm:cfc} it follows that one of the
two Finsler metrics has constant curvature if and only if the
corresponding Weyl-type curvature tensor vanishes. Now, using formula
\eqref{pw0} the other Weyl-type curvature tensor vanishes, and hence
the other Finsler metric has constant curvature, if and only
if $\delta_SP=0$.

It remains to prove the equivalence of the two conditions
$\delta_SP=0$ and $d_hd_JP=0$. From formula \eqref{elform} we have
  $\delta_SP=d_J{\mathcal L}_SP-2d_hP$ and hence
  $d_J\delta_SP=-2d_Jd_hP=2d_hd_JP$. Therefore $\delta_SP=0$ implies
  $d_hd_JP=0$.

For the converse we assume $d_hd_JP=0$. Using the commutation formula
for the two derivations $i_S$ and $d_h$, see \cite[Appendix A]{GM00},
we have 
\begin{eqnarray*}
i_Sd_hd_JP & = & -d_hi_Sd_JP+ {\mathcal L}_{hS}d_JP + i_{[h,S]}d_JP \\
& = & -d_hP+ {\mathcal L}_{S}d_JP - d_vP = \delta_SP. 
\end{eqnarray*}
In the above formula we did use that $P$ is a $1$-homogeneous function
and hence $i_Sd_JP=d_{JS}P={\mathcal C}(P)=P$, $hS=S$ and $J\circ[h,S]=-v$.
\end{proof}

A function $P$, positively $1$-homogeneous in the fiber coordinates
that satisfies one of the two equivalent conditions i) or ii) of
Theorem \ref{thm:FBT} is called a \emph{Hamel function}. 

If one of the Finsler metrics in Theorem \ref{thm:FBT}  is the
Euclidean one then we obtain the following characterisation for
projectively flat Finsler metrics of constant flag curvature.
\begin{cor} \label{cor:pf}
A projectively flat Finsler metric $F$ has
constant curvature if and only if the projective factor $P$ satisfies either one of the
following equivalent Hamel conditions:
\begin{itemize}
\item[i)] $\delta_{S_0}P=0$,
\item[ii)] $d_{h_0}d_JP=0$,
\end{itemize}
where $S_0$ is the flat spray and $h_0$ the corresponding horizontal projector.
\end{cor}
Corollary \ref{cor:pf} corresponds to \cite[Theorem 1.2]{LS18}.

According to Proposition \ref{prop:drdj}, in Riemannian geometry, the
Hamel condition \eqref{drdj3} means that for the linear projective
factor $P$, the basic $1$-form $d_JP$ is closed. We can use these
aspects and Theorem \ref{thm:FBT} to provide a Finslerian proof of
classical Beltrami Theorem.

\begin{thm} \label{thm:RBT} \emph{(Beltrami Theorem)}
Consider two projectively related
Finsler metrics $F$ and $\widetilde{F}$, each of them being reducible to a Riemannian
metric. Then, $F$ has constant curvature if and only if
$\widetilde{F}$ has constant curvature.
\end{thm}

\begin{proof}
The only think we have to prove is that the projective factor $P$
satisfies the condition $d_hd_JP=0$ of Theorem \ref{thm:FBT}. It
follows from the fact that $P$ is linear in the fiber coordinates and
satisfies formula \eqref{drdj3} of Proposition \ref{prop:drdj}.
\end{proof}

\section{Examples}

In this section we use some examples to test the Hamel conditions i) and
ii) of Theorem \ref{thm:FBT} for the projective factor of two
projectively related Finsler metrics.

\subsection{Projectively flat Finsler metric of non-constant curvature} 

We provide now an example of a projectively flat Finsler metric, which does not have
constant curvature. In view of Theorem \ref{thm:FBT} this can be
explained by the fact that the projective factor does not satisfy the
Hamel equation $\delta_{S_0}P=0$.  We consider the following Numata
metric, \cite[\S 3.9]{BCS00}, on the unit Euclidean ball $B^n(1)$:
\begin{eqnarray*}
F(x,y)=|y|+\langle x,y\rangle.
\end{eqnarray*}
The metric written above is a projectively flat Finsler metric, whose geodesic spray is given by:
\begin{eqnarray*}
S=S_0-\frac{|y|^2y^i}{|y|+\langle x , y \rangle}\frac{\partial }{\partial y^i},
\end{eqnarray*}
where $S_0$ is the flat spray, the geodesic spray of the Euclidean
metric.

The projective factor is given by 
\begin{eqnarray*}
P=\frac{S_0(F)}{2F}= \frac{1}{2}\frac{|y|^2}{|y|+\langle x , y \rangle}.
\end{eqnarray*}
By a direct computation we have 
\begin{eqnarray*}
\delta_{S_0}P=	\left\{\frac{|y|^2}{2F^3}\frac{\partial}{\partial x^i}\left(|x|^2|y|^2
				-\langle x, y\rangle^2\right)\right\}dx^i\neq 0.
\end{eqnarray*}
Since $S_0$ is the flat spray and the projective factor does not
satisfy the Hamel equation $\delta_{S_0}P=0$, it follows by Theorem
\ref{thm:FBT} that the Numata metric does not have constant
curvature. This can be checked independently, since its scalar flag
curvature is given by 
\begin{eqnarray*}
\kappa(x,y)=\frac{3}{4}\frac{|y|^4}{\left({|y|+\langle x , y \rangle}\right)^4}.
\end{eqnarray*} 

\subsection{Projectively flat Finsler metric of constant
  curvature}

In this example, we consider a projectively flat Finsler metric $F$ with
the projective factor satisfying the Hamel equation $\delta_{S_0}P=0$. Using the Finslerian version of Beltrami Theorem
\ref{thm:FBT} it follows that the Finsler metric $F$ has constant
curvature. 

The Funk function on the Euclidean ball $B^n(1/2)$, \cite[\S 2.3]{Shen01}, 
\begin{eqnarray*}
F(x,y)=\frac{\sqrt{(1-4|x|^2)|y|^2+4\langle x, y \rangle^2}}{1-4|x|^2}+2\frac{\langle x, y \rangle}{1-4|x|^2}
\end{eqnarray*}
is a Finsler metric, projectively related  to the Euclidean metric. Its geodesic spray
is given by
\begin{eqnarray*}
S=S_0-2P\mathcal{C},
\end{eqnarray*}
where $S_0$ is the flat spray. The projective factor $P$ is given by 
\begin{eqnarray*}
P=\frac{S_0(F)}{2F}=F. \end{eqnarray*}
Recall that $F$ is projectively flat, therefore it satisfy the
Hamel equation $\delta_{S_0}F=0$. It follows that the
projective factor $P=F$ satisfies also the equation $\delta_{S_0}P=0$
and due to the Finslerian version of Beltrami Theorem
\ref{thm:FBT} it follows that the Finsler metric $F$ has constant
curvature.  One can independently check that the Finsler metric $F$
has constant flag curvature $\kappa=-1$.

\subsection{Projectively related Finsler metrics with linear
  projective factor}

In this example we consider two projectively related Finsler metric
with a linear projective factor $P$. If the first Finsler metric has
constant curvature, then the obstruction for the second Finsler
metric to be of constant curvature reduces to the fact that the
basic $1$-form $d_JP$ is closed. 

It is known that the Funk metric $F$, on the Euclidean unit ball $B^n(1)$, 
\begin{eqnarray*}
F(x,y)=\frac{\sqrt{|y|^2-(|x|^2|y|^2-\langle x, y\rangle^2)}}{1-|x|^2}+\frac{\langle x, y \rangle}{1-|x|^2},
\end{eqnarray*}
is a projectively flat Finsler metric of constant flag curvature $\kappa=-\frac{1}{4}$ with the geodesic spray:
\begin{eqnarray*}\label{spray-mfunk}
S=y^i\frac{\partial}{\partial x^i}-F(x,y)y^i\frac{\partial}{\partial y^i}.
\end{eqnarray*}
From \cite[Example 5.3]{Shen02}, we know that for any constant vector $a\in {\mathbb R}^n$, $|a|<1$, the function
\begin{eqnarray*}
\widetilde{F}(x,y)=\frac{\sqrt{|y|^2-(|x|^2|y|^2-\langle x,
    y\rangle^2)}}{1-|x|^2}+\frac{\langle x, y
  \rangle}{1-|x|^2}+\frac{\langle a,y\rangle}{1+\langle a,x\rangle}
\end{eqnarray*}
is a projective Finsler metric on $B^n(1)$.

The two Finsler metrics $F$ and $\widetilde{F}$ are projectively
related with the projective factor linear in the fiber coordinates:
\begin{eqnarray*}
P=\frac{S(\widetilde{F})}{2\widetilde{F}}=-\frac{1}{2}\frac{\langle
  a,y\rangle}{1+\langle a,x \rangle}. 
\end{eqnarray*} 
It follows that the basic $1$-form $d_JP$ is exact, 
\begin{eqnarray*}
d_JP=-\frac{1}{2}\frac{a_idx^i}{1+\langle a,x \rangle} = d\left(\ln(1+\langle a,x \rangle)^{-1/2}\right). 
\end{eqnarray*} 
Therefore, the obstruction ii) for the projective factor $P$ in Theorem
\ref{thm:FBT} is satisfied and hence the Finsler metric
$\widetilde{F}$ has constant curvature as well. This can be checked
directly, the flag curvature of $\widetilde{F}$ is
$\widetilde{\kappa}=-1$. 

\subsection*{Acknowledgments} We express our thanks to Vladimir
Matveev, Zhongmin Shen and J\'ozsef Szilasi for their comments and
suggestions on this work.

\end{document}